\documentclass[12 pt,oneside]{article}
\usepackage{amsmath}
\usepackage{amsxtra}
\usepackage{amscd}
\usepackage{amsthm}
\usepackage{amsfonts}
\usepackage{amssymb}
\usepackage{eucal}

\usepackage[margin=2.5 cm, reversemp]{geometry}

\usepackage{color}
\usepackage{diagrams}

\theoremstyle{definition}
\newtheorem{thm}{Theorem}[section]
\newtheorem{cor}[thm]{Corollary}
\newtheorem{conj}[thm]{Conjecture}
\newtheorem{lemma}[thm]{Lemma}

\newtheorem{defn}[thm]{Definition}

\newcommand\nc{\newcommand}
\nc\Span{\text{\rm Span}}
\nc\Id{\text{Id}}

\nc \cc {\mathbb{C}}
\nc \ff {\mathbb{F}}
\nc \nn {\mathbb{N}}
\nc \pp {\mathbb{P}}
\nc \qq {\mathbb{Q}}
\nc \rr {\mathbb{R}}
\nc \zz {\mathbb{Z}}

\nc\cA{\mathcal{A}}
\nc\cC{\mathcal{C}}
\nc\cD{\mathcal{D}}
\nc\cE{\mathcal{E}}
\nc\cF{\mathcal{F}}
\nc\cG{\mathcal{G}}
\nc\cH{\mathcal{H}}
\nc\cK{\mathcal{K}}
\nc\cM{\mathcal{M}}
\nc\cO{\mathcal{O}}
\nc\cP{\mathcal{P}}
\nc\cR{\mathcal{R}}
\nc\cS{\mathcal{S}}

\nc\fm{\mathfrak{m}}
\nc\fp{\mathfrak{p}}

\nc\Tr{\text{Tr}}
\nc\into{\hookrightarrow}

\nc\st{\text{ s.t. }}
\nc\intense[1]{\textcolor[rgb]{1.00,0.00,0.00}{\textbf{#1}}}

\renewcommand{\(}{\left(}
\renewcommand{\)}{\right)}

\nc\Mat{\text{\rm Mat}}
\nc\GL{\text{\rm GL}}
\nc\SU{\text{\rm SU}}
\nc\SO{\text{\rm SO}}
\nc\SL{\text{\rm SL}}
\nc\Sp{\text{\rm Sp}}
\nc\EL{\text{\rm EL}}

\nc\GEM{\text{\rm GEM}}
\nc\Alt{\text{\rm Alt}}
\nc\Sym{\text{\rm Sym}}

\renewcommand{\(}{\left(}
\renewcommand{\)}{\right)}

\nc\inject{\hookrightarrow}

\nc{\mattwo}[4]{\left[\begin{array}{cc} #1  & #2\\  #3 & #4 \\ \end{array} \right]}
\nc{\matthree}[9]{\left[\begin{array}{ccc} #1  & #2 & #3\\  #4 & #5 & #6 \\ #7 & #8 & #9 \\ \end{array} \right]}
\nc{\vecttwo}[2]{\left[\begin{array}{c} #1 \\ #2 \\ \end{array} \right]}
\nc{\vectthree}[3]{\left[\begin{array}{c} #1 \\ #2\\  #3 \\ \end{array} \right]}

\nc{\del}{\partial}
\nc\onto{\twoheadrightarrow}

\nc\const{\text{const}}

\nc\rrp{\rr P}
\nc\ul{\underline}
\nc\ol{\overline}
\nc\uline{\underline}
\nc\oline{\overline}
\nc\oset{\overset}
\nc\uset{\underset}

\nc\heart{\heartsuit}
\nc\spade{\spadesuit}
\nc\club{\clubsuite}

\nc\marg[1]{\marginnote{\boxed{\text{#1}}}}
\nc\margq[1]{\marginnote{\textcolor[rgb]{1.00,0.00,0.00}{#1}}}

\nc\Char{\text{Char}}
\nc\Frac{\text{Frac}}
\nc\wo{\backslash}
\nc\diag{\text{diag}}
\nc\wtl{\widetilde}
\nc\nsubgp{\triangleleft}

\nc\Cay{\text{Cay}}
\nc\Hom{\text{Hom}}
\nc\Gp{\text{Gp}}
\nc\Set{\text{Set}}
\nc\la{\langle}
\nc\ra{\rangle}

\nc\wht{\widehat}
\nc\ddx[2]{\frac{\partial {#1}}{\partial {#2}}}
\nc\dddx[3]{\frac{\partial^2 {#1}}{\partial {#2}\partial{#3}}}
\nc\mult{\text{mult}}
\nc\supp{\text{supp}}
\nc\sign{\text{sign}}

\nc\meas{\text{meas}}
\nc\tmax{\text{max}}
\nc\red{\text{red}}
\nc\dhamm{d_{\text{Hamm}}}
\nc\tr{\text{tr}}

\title{Finite Part of Operator $K$-Theory for Groups with Rapid Decay}
\author{}

\begin{document}
\maketitle

\begin{abstract} In this paper we study the part of the $K$-theory of the reduced $C^*$-algebra arising from torsion elements of the group, and in particular we study the pairing of $K$-theory with traces and when traces can detect certain $K$-theory elements. In the case of groups with Property RD, we give a condition on the growth of conjugacy classes that determines whether they can be detected. Moreover, in the case that they can be detected, we show that nonzero elements in the part of the $K$-theory generated by torsion elements are not in the image of the assembly map $K^G_0(EG) \to K_0(C^*G)$. One application of this is a lower bound for the structure groups of certain manifolds.
\end{abstract}

\section{Introduction}

Let $G$ be a group. An element $g$ is said to have order $d$ for positive integer $d$ if $g^d=1$ and $g^k \neq 1$ for any $0 < k <d$. If no such $d$ exists, we say that $g$ has order $\infty$. We say an element is torsion if it has finite order.

For a torsion element $g$ of order $d$, we define a corresponding idempotent $p$ as
\[p = \frac{1+g+g^2+\cdots+g^{d-1}}{d} \in \cc G.\]
It is easy to see that this is an idempotent, so we can view it in $K_0(\cc G)$. Let $C^*_{\max}G$ denote the maximal group $C^*$-algebra and $C^*_{\red}G$ denote the reduced group $C^*$-algebra. Then one can also view the above idempotent in $K_0(C^*_{\red}G)$ or $K_0(C^*_{\max}G)$. The question we would like to study is whether for idempotents $p_1,\ldots p_k$ corresponding to torsion elements $g_1,\ldots, g_k$ of distinct orders $d_1,\ldots d_k>1$, we have that $1,p_1,\ldots p_k$ are linearly independent in $C^*_{\red}G$ and $C^*_{\max}G$.

The study of idempotents arising from torsion elements was introduced in \cite{WY}. In that paper, Weinberger and Yu showed that if the group is finitely embeddable in Hilbert spaces, then the $p_i$ are linearly independent in $K_0(C^*_{\tmax}(G))$. They conjectured that the same holds for all groups. They also showed that the rank of the part of $K_0(C^*_{\tmax}(G))$ generated by these torsion elements bounds the non-rigidity for $M$. More precisely, it gives a lower bound for the rank of the structure group $S(M)$, that is, the abelian group of equivalence classes of pairs $(f,M')$, where $M'$ is a compact, oriented manifold, and $f:M' \to M$ is an orientation preserving homotopy equivalence.

Our approach to proving the linear independence of $p_i$ is to find $k+1$ traces on the algebra such that the $(k+1) \times (k+1)$ matrix that arises from evaluating the traces on $1,p_1,\ldots p_k$ has non-zero determinant. If there are such traces, then we say that $1,p_1,\ldots p_k$ can be distinguished by traces. It is easy to see that being distinguished by traces implies linear independence in $K_0$. For groups with rapid decay property (property RD), we prove a necessary and sufficient condition for the $1,p_1,\ldots p_k \in K_0(C^*_{\red}G)$ to be distinguished by traces, namely that they can be distinguished by traces if and only if the relevant conjugacy classes grow polyonimally. We discuss some examples of both our positive and our negative results. The negative results are particularly striking because they demonstrate a phenomenon in reduced group $C^*$-algebras that is very different from the maximal $C^*$-algebra case.

The author would like to thank Prof. Guoliang Yu for his helpful advice and guidance.




\section{Results}	

Consider a group $G$ with torsion elements $g_1,g_2,\ldots, g_k$ of orders $d_1,\ldots d_k$ respectively, where the $d_i$ are distinct postive integers. For each $g_i$, we can consider the element
\[p_i= \frac{1+g_i+g_i^2+\cdots+g_i^{d_i-1}}{d_i} \in \cc G.\]
Then $p_i$ is an idempotent, so we may consider it in the $K_0$ groups, $K_0(C^*_{\tmax}(G))$ and $K_0(C^*_{\red}(G))$. We would like to explore whether
\[p_0=1,p_1,\ldots p_k\]
are linearly independent in these $K_0$ theories.

We call the subgroup of $K_0(C^*_{\max}G)$ (resp. $K_0(C^*_{\red}G)$) generated by the $p_i$ the finite part of $K_0(C^*_{\max}G)$ (resp. $K_0(C^*_{\max}G)$).

It was conjectured by Weinberger and Yu in \cite{WY} that the $p_i$ are linearly independent in $K_0(C^*_{\max}G)$. More precisely:
\begin{conj} For a group $G$ with torsion elements $g_1,g_2,\ldots, g_k$ of orders $d_1,\ldots d_k$ respectively, where the $d_i$ are distinct positive integers. Let $p_i$ denote the idempotent corresponding to $g_i$ as above. Then the $1,p_1,\ldots p_k$ are linearly independent in $K_0(C^*_{\max}G)$. Moreover, any nonzero element in the subgroup generated by $p_1,\ldots p_k$ is not in the image of the assembly map $K^G_0(EG) \to K_0(C^*_{\max}(G))$, where $EG$ is the universal space for proper and free $G$-action.
\label{WY conj}
\end{conj}

The assembly map is the map $K^G_0(EG) \to K_0(C^*_{\max}(G))$ studied in \cite{BFJR} and \cite{BCH}.

To get at this question, we consider whether there are traces $\tau_0,\tau_1,\ldots,\tau_k:C^*G: \to \cc$ such that the matrix $A_{ij}=\tau_i(p_j)$ has non-zero determinant. If this were true, then we know that the $p_i$ are linearly independent, and we say that they are distinguishable by traces. If there are no such $\tau_i$, then the $p_i$ may still be linearly independent, but there is no way to prove this by traces, so we say that they are not distinguishable by traces.

Note that if idempotents can be separated by traces on $C^*_\red G$, then they can also be separated by traces on $C^*_\tmax G$, because there is a natural map $C^*_\tmax G \to C^*_{\red}G$, so the traces on $C^*_{\red}G$ induce traces on $C^*_\tmax G$ in a way that does not affect the values on the idempotents. However, as we shall see, the converse is not true--indeed we shall give an example in which the idempotents can be separated by traces in $C^*_\tmax$, but not in $C^*_\red$.

One other notion we will need is that of property RD. Recall that a length function on a group $G$ is a function $l:G \to \zz^{ \geq 0}$, such that $l(fg) \leq l(f)+l(g)$, $l(g^{-1})=l(g)$, and $l(e)=0$, such that for $S \subset \zz^{ \geq 0}$ finite $l^{-1}(S)$ is also finite. Given such a length function on a group, there is a norm $\|\cdot\|_{H^s}$ on $\cc G$ given by
\[\|\sum c_g g\|_{H^s}^2 = \sum |c_g|^2(1+l(g))^{2s}.\]
Let $H^s(G)$ denote the completion of $\cc G$ with respect to this norm. Then we say $G$ has rapid decay property, or property RD if there is a length function on it and a constant $C$ such that
\[\|\sum c_g g\|_{\red} \leq C\|\sum c_g g\|_{H^s}.\]

In this case, there is a natural map $H^sG \to C^*_{\red}G$ obtained by taking the identity $\cc G \to \cc G$.

Note that for $t>s$, $\|x\|_{H^t}>\|x\|_{H^s}$ and we have natural map $H^t(G) \to H^s(G)$. So if we have property RD for $s$, then we also have the inequality for any $t>s$.

Jolissaint introduced this property in \cite{Joli1}. He also showed in \cite{Joli-K} that for such groups the map $H^sG \to C^*_{\red}G$ induces and isomorphism of $K$-theory $K^0(H^sG) \to K^0(C^*_{\red}G)$.

Let us define some notation. For $g_i$, let $C(g_i) \subset G$ denote the conjugacy class of $g_i$. 

We have the following results:

\subsection{Positive Results}

Let $G$ be a group with Property RD. Suppose $C(g_i)$ has polynomial growth, that is, for $C(g_i)_l= \{g| g \in C(g_i), l(g)=l\}$ and $n_{i,l}=|C(g_i)_l|$, we have $n_{i,l}<P(l)$ for some polynomial $P$, then the $p_i$ can be separated by traces on $H^s(G)$ for some $s$. More precisely, we have:

\begin{lemma} Let $G$ be a group with Property RD and suppose $C(g_1)$ has polynomial growth. Then the trace $\tau:\cc G \to G$ given by $\tau\(\sum c_g g\) = \sum_{g \in C(g_1)}c_g$ extends to $H^s(G)$ for sufficiently large $s$.
\end{lemma}
\begin{proof} Notice that
\begin{equation}\left|\sum_{g \in C(g_1)} c_g\right|  \leq \sum_l \left| \sum_{g \in C(g_1)_l} c_g\right| \leq \sum_l \sqrt{n_{1,l}}\sqrt{\sum_{g \in C(g_1),l}c_g^2}
\leq C\sqrt{\sum_l n_{1,l} \(\sum_{g \in C(g_1)_l} c_g^2\)l^2}\label{rd poly 1}\end{equation}
for some consant $C$. Here, the second inequality follows from the Power Mean Inequality, which says $\frac{\sum_{i=1}^n a_i}{n} \leq \sqrt{\frac{\sum a_i^2}{n}}$, and the third inequality follows from Cauchy Schwarz inequality, which implies that for sequence $a_l$, $\sum a_l \leq \sqrt{\(\sum a_l^2l^2\)\(\sum \frac{1}{l^2}\)} \leq \frac{\pi}{\sqrt{6}}\sqrt{\(\sum a_l^2l^2\)}$.

But because $n_{1,l}$ is polynomial growth, the right hand side of \eqref{rd poly 1} is bounded by the $H^s$ norm for some $s$, so the trace can be extended, as desired.
\end{proof}

\begin{thm} Let $G$ be a group with Property RD and suppose that the $C(g_i)$ all have polynomial growth. Then the $p_i$ can be separated by traces on $H^s$. Thus, they are linearly independent in $K_0(H^sG)$. Moreover, they are linearly independent in $K_0(C^*_{\red}G)$, and in $K_0(C^*_{\red}G)$, nonzero elements in the subgroup generated by $p_i$ cannot lie in the image of the assembly map. Thus, groups with property RD where all the conjugacy classes of torsion elements grow polynomially satisfy Conjecture \ref{WY conj}.
\label{thm lin indep}
\end{thm}
\begin{proof} Consider $\tau_0:H^s(G) \to \cc$ given by $\tau_0\(\sum c_g g \) = c_e$. This is clearly a bounded map. For $i=1,\ldots k$ consider $\tau_i:H^s(G) \to \cc$ given as above by
\[\tau_i(\sum c_gg) = \sum_{g \sim g_i} c_g.\]
Then $\tau_0(p_i)=\frac{1}{d_i}$. Also note that for $d_i>d_j$, for any $t$, we have $\deg(g_j^t) \leq \deg(g_j)<\deg(g_i)$, so $g_j^t$ and $g_i$ cannot be similar for any $j$. Thus for $d_i>d_j$, we have $\tau_i(p_j)=0$.

Also note that $\tau_i(p_i) \geq \frac{1}{d_i}$, because $\tau_i(g_i^s)$ is either $0$ or $1$, and it is $1$ for $s=1$.

So if we suppose that $d_1<d_2,\ldots$, we get that matrix $\tau_i(p_j)$ is upper triangular with all diagonal elements non-zero, so its determinant is nonzero, as desired.

Jolissaint showed in \cite{Joli-K} that for property RD groups, the map $H^sG \to C^*_{\red}G$ is an isomorphism on the level of $K$-theory. That is, property RD implies $K^0(H^s(G)) \simeq K^0(C^*_\red(G)$. Thus, the $p_i$ are also linearly independent in $C^*_\red(G)$. 

Let us show that the non-zero elements in the subgroup generated by $p_i$ are not in the image of the assembly map $K_0^G(EG) \to K_0(C^*_{\red}G)$. For a Hilbert space $H$, let $\cS$ denote the ring of Schatten class operators, that is operators $T:H \to H$ such that there exists $p \geq 1$ such that $\tr(T^*T)^{p/2}<\infty$, where the trace $\tr$ is the sum of the singular values. Let $\cS G$ denote the group algebra of $G$ with coefficients in $\cS$. As in \cite{WY} and \cite{Yu}, let $H_0^{orG}(EG, \mathbb{K}({\cal S})^{-\infty})$ be the generalized $G$-invariant homology theory associated to the non-connective algebraic K-theory spectrum $ \mathbb{K}(\cS)^{-\infty})$. Then there is commutative diagram
\begin{diagram}
H^{or G}_0(EG, \mathbb{K}(\cS)^{-\infty}) & \rTo^A & K_0(\cS G)\\
\dTo^{\simeq} & \ruTo & \dTo \\
K^G_0(EG) & \rTo^{\mu} & K_0(C^*_{\red}G)\\
\end{diagram}
There is a homomorophism $j:\cc G \to \cS G$ given by $j(a)=P_0a$ where $P_0$ is a rank 1 projection. We may consider the finite part of $K_0(\cS G)$ to be the subgroup generated by $j_*(p_i)$ where the $p_i$ are defined as above. To show that the finite part of $K_0(C^*_{\red}G)$ is not in the image of $\mu$ it suffices to show that the finite part of $K_0(\cS G)$ is not in the image of $H_0^{orG}(EG, \mathbb{K}({\cal S})^{-\infty})$.

For an algebra $\cA$, we say that $\tau_n:\cA^{\times n+1} \to \cc$ is a multitrace if it satisfies
\[\tau_n(a_0,a_1,\ldots a_n)=(-1)^n\tau_n(a_1,a_2,\ldots, a_n,a_0)\]
and
\[\sum_{i=0}^n (-1)^i\tau_n(a_0, \ldots, a_ia_{i+1}, \ldots a_{n+1})+(-1)^{n+1}\tau_n(a_{n+1}a_0, a_1,\ldots a_n)=0.\]

For $g$ of finite order, consider the trace $\tr_g$ on $\cS_1 G$ given by $\tr_g(\sum s_{\gamma}\gamma)=\sum_{\gamma \in C(g)}\tr (s_\gamma)$, where $\tr$ is the standard trace on $\cS_1$. Then for $n = 2k$, consider $\tau_{n,g}$ on $S_n(G)$ given by
\[\tau_{n,g}(a_0,\ldots a_n)=\tr_{g}(a_0a_1\cdots a_n).\]

It was shown in \cite{WY} (see the proof of Lemma 3.2 of \cite{WY}) that for element $p \in K_0(\cS G)$ in the image of $A$, we have that such $\tau_{n,g}$, $\tau_{n,g}(p,p,\ldots , p)=0$.

For groups with property RD, we have commutative diagram
\begin{diagram}
K_0^G(EG) & & \\
\dTo & \rdTo & \\
 K_0(H^\infty G) & \rTo^\simeq & K_0(C^*_{\red}G)\\
\end{diagram} 
where $H^\infty G$ is the intersection of the $H^sG$, where the vertical maps are assembly maps.

Again as in \cite{WY}, according to Connes' cyclic cohomology  $$j^*\tau_{n,g}= S^k\tau_g,$$ where $\tau_g: {\mathbb C}G \to {\mathbb C}$ is defined by $\tau_g(\sum c_\gamma\gamma) =\sum_{\gamma\in C(g)} c_\gamma$. In \cite{WY} it was shown that for any element $x$ in the image of the assembly map $A:H^{or G}_0(EG, \mathbb{K}(\cS)^{-\infty})\to K_0(\cS G)$, we have $\la \tau_{n,g}, x \ra = 0$ for the $\tau_{n,g}$ as defined above, but in this case, we know that there is a trace $\tau_{g_i}$ (where $\{g_i\} $ is as in Conjecture \ref{WY conj}) of this form that does not vanish on $x$, so $x$ cannot be in the image.

Now we show the conjecture, which is the same two statements, but for $C^*_{\max}G$ instead of $C^*_{\red}G$. For the first, note that there is a natural map $C^*_{\max}(G) \to C^*_{\red}(G)$, so linear independence in $C^*_{\red}G$ implies linear independence in $C^*_{\max}(G)$. For the second, note that the assembly map $K_0^G(EG) \to K^0(C^*_{\red}G)$ factors through $K_0(C^*_{\max}G)$, so if the elements considered are in the image $K_0^G(EG)$ in $K_0(C^*_{\max}G)$, then they are also in the image in $K^0(C^*_{\red}G)$.

Thus, we have shown Conjecture \ref{WY conj} for Property RD groups whose conjugacy classes of torsion elements have polynomial growth.
\end{proof}

Note that we have also established the conjecture for groups that are residually contained in the above class:

\begin{defn} Let $\cR$ be the class of groups $G$ with property RD such that for any torsion element $g$, $C(g)$ has polynomial growth. We say that a group $G$ is residually in $\cR$ if for any finite subset $F \subset G$ there is a group $H \in \cR$ and a homomorphism $\phi:G \to H$ such that $\phi$ is injective on $F$. 
\end{defn}

\begin{thm} Any group that is residually in $\cR$ satisfies Conjecture \ref{WY conj}.
\end{thm}
\begin{proof} Let $g_1,\ldots g_k$ be the torsion elements of order $d_1,\ldots d_k$ in question, and let $p_1,\ldots p_k$ be the corresponding idempotents. Let $F = \{g_i^{t_i}|0 \leq t_i<d_i\}$. Then note that the $\phi(g_i)$ are torsion elements of order $d_i$ in $H$, and their corresponding idempotents in $\cc H$ are $\phi(g_i)$, where $\phi$ denotes the induced map $\phi: \cc G \to \cc H$. But note that any unitary representation of $H$ is also a unitary representation of $G$, so $\phi$ also induces a map $C^*_{\max}G \to C^*_{\max}H$. By the above corollary, $H$ satisfies the conjecture, so for $p_i \in K_0(C^*_{\max}G)$, we have $\phi(p_i) \in K_0(C^*_{\max}H)$ are linearly independent. But this means that the $p_i$ are linearly independent in $K_0(C^*_{\max}G)$ as well.

The fact that nonzero elements in the subgroup generated by $p_i$ are not in the image of the assembly map follows from the functoriality of the assembly map.
\end{proof}

This theorem has consequences related to structure groups of manifolds. In particular we have

\begin{cor} Let $M$ be a compact manifold of dimension $4k-1$ with $k>1$ an integer and $\pi_1(M)=G$, and $G$ is in the class $\cR$ described above. Then the rank of the structure group $S(M)$ is bounded below by the number of distinct orders of torsion elements of $G$.
\end{cor}
\begin{proof} This follows from Theorem 1.5 in \cite{WY}, which states that this corollary holds for any $G$ satisfying Conjecture \ref{WY conj}
\end{proof}

It is reasonable to conjecture that a similar state to Conjecture \ref{WY conj} holds:

\begin{conj} For a group $G$ with torsion elements $g_1,g_2,\ldots, g_k$ of orders $d_1,\ldots d_k$ respectively, where the $d_i$ are distinct positive integers. Let $p_i$ denote the idempotent corresponding to $g_i$. Then the $1,p_1,\ldots p_k$ are linearly independent in $K_0(C^*_{\red}G)$. Moreover, any nonzero element in the subgroup generated by $p_1,\ldots p_k$ is not in the image of the assembly map $K^G_0(EG) \to K_0(C^*_{\max}(G))$, where $EG$ is the universal space for proper and free $G$-action.
\end{conj}

\subsection{Negative Results}

Now let us investigate the case where the group has Property RD, but some of the conjugacy classes do not have polynomial growth. We will show that all traces must be zero on these conjugacy classes. 

\begin{lemma} Let $G$ be a group of Property RD and Let $h \in G$ be an element (which does not necessarily have to have finite order). Let $C(g)$ be its conjugacy class and let $C(g)_l \subset C(g)$ be those elements of length $l$, and let $n_l=|C(g)_l|$. Assume that $n_l$ is does not grow polynomially, that is, for any polynomial $P$, there are an infinite number of $l$ such that $n_l>P(l)$. Then any trace $\tau:C^*_\red G \to \cc$ must have $\tau(h)=0$.
\end{lemma}
\begin{proof} We proceed by contradiction. Suppose we have $\tau$ a trace such that $\tau(h) \neq 0$. By scaling it, we can assume $\tau(h)=1$. Then because $\tau$ is a trace, it is constant on conjugacy classes, so for any $h \in C(g)$, we have $\tau(h)=1$.

Since $n_l$ is super-polynomial, we may consider an increasing sequence $l_1,l_2,\ldots$ such that $n_{l_i}>(1+l_i)^{4i}$. That is, $n_{l_1} > (1+l_1)^{4}$, $n_{l_2} > (1+l_2)^{8}$, etc. 

Now consider the element $x=\sum x_h h$ with
\[x_h = \begin{cases}n_{l(h)}^{-5/8} & l(h) \in \{l_1,l_2,\ldots\} \text{ and } h \in C(g) \\ 0 & \text{else}\end{cases}.\]

Then note that for the Sobolev norm $\|\sum c_h h\|_{H^s}^2=\sum |c_h|^2(1+l(h))^{2s}$, we have that the above element is in $H^s$ with
\[\|\sum x_h h\|_{H^s} = \sum_i n_{l_i} \cdot n_{l_i}^{-10/8}(1+l_i)^{2s} < \sum_i (1+l_i)^{-i}(1+l_i)^{2s}\]
which clearly converges.

In fact, $x_N = \sum_{l(h)<N}x_hh$ converges to $x$ in the $H^s$ norm. Therefore, this is also true in the $C^*_rG$, because property RD.

However, I claim that for any $\tau$ with $\tau(h)=1$ on $C(g)$, $\tau$ does not converge on $x_N$, therefore $\tau$ is not a bounded linear functional on $H^s$. Note that the value would be
\[\sum_{i \text{ with }l_i<N} n_{l_i} \cdot n_{l_i}^{-5/8}=\sum_{i \text{ with }l_i<N} n_{l_i}^{3/8}\]
but $n_{l_i}>1$ for all $i$, so clearly this sum diverges as $N \to \infty$.

Thus, $\tau$ cannot be a bounded linear functional.
\end{proof}

Applying this lemma we have the following theorem:

\begin{thm} Let $G$ be a group with Property RD and let $g_1 \in G$ be a torsion element such that the conjugacy classes of $g_1^t$ do not grow polynomially for any $t$. Then in $C^*_\red(G)$ we cannot distinguish $1,p_1$ by traces. 
\end{thm}
\begin{proof} Applying the above lemma to each $g_i^t$ for $0<t<d$, we get that $\tau(g_i^t)=0$ for all $t$, which means that $\tau(p_i)=\tau(\frac{1}{d_i})$, so for any trace $\tau$, we have $\tau([p_i])=\frac{1}{d_i}\tau([1])$, so we cannot show that $1$ and $p_1$ are linearly dependent using traces.
\end{proof}

Note that a similar result does not hold for $C^*_\tmax G$. In particular, for torsion element $g \in G$ with order $d$, and $p$ the corresponding idempotent, we can always distinguish $[1]$ and $[p]$ using the traces $\sum c_g g \mapsto c_e$ and $\sum c_g g\mapsto \sum c_g$, where the latter is a bounded map because it is the trace of the trivial action of $G$ on $\cc$.

One example satisfying the above conditions is $G = <x,y>/x^3$ with $g_1 = x$. Then $G$ is a hyperbolic group, so it has Property RD. Also, $x$ and $x^2$ both have exponential growth conjugacy classes. In particular, the conjugacy classes contain $gxg^{-1}$ for 
\[g=yx^{a_1}yx^{3-a_1}yx^{a_2}yx^{3-a_2}y\cdots yx^{a_l}yx^{3-a_l}y,\]
where $a_i \in \{1,2\}$. Then $g$ has length $5l+1$ and $gxg^{-1}$ has length $10l+3$, but there are $2^l$ choices of $g$. Thus, applying the above corollary, we see that $[1]$ and $[p_1]$ are not distinguishable by traces.

We also have another corollary to the above lemma:

\begin{thm} Suppose that $G$ is a Property RD group such that none of its conjugacy classes grow polynomially. Then the only trace on $C^*_\red G$ is $\sum c_g g \mapsto c_e$.
\end{thm}

\end{document}